\documentclass[a4]{amsart}

\usepackage{amssymb}
\usepackage[all]{xy}
\usepackage{amsmath, amsfonts, amsthm , amssymb} 
\usepackage{mathrsfs}
\usepackage{color}
\usepackage{hyperref}
\usepackage[normalem]{ulem}
\usepackage{colonequals}
\usepackage{comment}

\input xypic 
\usepackage{tikz} 
\usepackage{tikz-cd}
\usepackage{tikz-3dplot}  

\usepackage{romannum}

\newtheorem{theorem}{Theorem}[section]
\newtheorem{proposition}[theorem]{Proposition}
\newtheorem{lemma}[theorem]{Lemma}
\newtheorem{corollary}[theorem]{Corollary}

\theoremstyle{definition}

\newcommand{\sk}{\ensuremath{\mbox{skel}\,}}

\newcommand{\cohlgy}[1]{\ensuremath{H^{*}(#1)}}

\newcounter{bean}

\newcommand{\namedright}[3]{\ensuremath{#1\stackrel{#2}
 {\longrightarrow}#3}}
\newcommand{\nameddright}[5]{\ensuremath{#1\stackrel{#2}
 {\longrightarrow}#3\stackrel{#4}{\longrightarrow}#5}}
\newcommand{\namedddright}[7]{\ensuremath{#1\stackrel{#2}
 {\longrightarrow}#3\stackrel{#4}{\longrightarrow}#5
  \stackrel{#6}{\longrightarrow}#7}} 

\newcommand{\larrow}{\relbar\!\!\relbar\!\!\rightarrow}
\newcommand{\llarrow}{\relbar\!\!\relbar\!\!\larrow}

\newcommand{\lnamedright}[3]{\ensuremath{#1\stackrel{#2}
 {\larrow}#3}}

\newcommand{\lnamedddright}[7]{\ensuremath{#1\stackrel{#2}
 {\larrow}#3\stackrel{#4}{\larrow}#5
  \stackrel{#6}{\larrow}#7}}
\newcommand{\llnamedright}[3]{\ensuremath{#1\stackrel{#2}
 {\llarrow}#3}}
\newcommand{\llnameddright}[5]{\ensuremath{#1\stackrel{#2}
 {\llarrow}#3\stackrel{#4}{\llarrow}#5}}

\newcommand{\Q}{\mathbb{Q}}

\newcommand{\Z}{\mathbb{Z}}

\newif\ifdebug
\debugtrue

\begin{document}
\title[Homotopy Rigidity]{Homotopy rigidity for quasitoric manifolds over a product of simplices}

\author{Xin Fu}
\address{Shanghai Institute for Mathematics and Interdisciplinary Sciences, Shanghai 200433, China} 
\email{x.fu@simis.cn} 
\author{Tseleung So} 
\address{Institute of Mathematical Science, Pusan National University, Busan 46241, Republic of Korea} 
\email{larry.so.tl@gmail.com} 
\author{Jongbaek Song} 
\address{Department of Mathematics Education, Pusan National University, Busan 46241, Republic of Korea} 
\email{jongbaek.song@pusan.ac.kr} 
\author{Stephen Theriault}
\address{Mathematical Sciences, University of Southampton, Southampton SO17 1BJ, UK} 
\email{S.D.Theriault@soton.ac.uk} 



\begin{abstract}
We consider a variation of the Cohomological Rigidity Problem: whether two quasitoric manifolds whose cohomology rings are isomorphic are homotopy equivalent.
We show that if two quasitoric manifolds have 
isomorphic cohomology rings and the associated polytope of one of them is a product of simplices, then they are homotopy equivalent after localizing away from 
an explicit list of primes.
\end{abstract}

\maketitle 

\pagenumbering{arabic}

\section{Introduction} 
\label{sec:intro} 

Let $P$ be a simple polytope of dimension $n$ with $m$ facets. There is an associated moment-angle manifold $\mathcal{Z}_{P}$ with an 
action of the $m$-torus $T^{m}$. If $T^{m-n}$ is a freely acting sub-torus, the quotient 
\[M=\mathcal{Z}_{P}/T^{m-n}\] 
is a \emph{quasitoric manifold}. These manifolds include constructions such as projective spaces, Hirzebruch surfaces and Bott manifolds, and play an important role in symplectic geometry, complex geometry, toric topology and combinatorics.

The Cohomological Rigidity Problem in toric topology~\cite{MS} asks whether two quasitoric manifolds 
with isomorphic cohomology rings are diffeomorphic or homeomorphic. This has been proved in several special cases: 
four-dimensional quasitoric manifolds~\cite{OrRa, Fre}, Bott manifolds~\cite{CHJ}, certain 
generalized Bott manifolds~\cite{CMS-tr}, quasitoric manifolds with the second Betti number equal to~$2$~\cite{CPS}, and 6-dimensional quasitoric manifolds associated with Pogorelov polytopes~\cite{BEMPP} or a $3$-cube~\cite{Ha}. 

Since cohomology is a homotopy invariant, the Cohomological Rigidity Problem may innately be considered with respect to homotopy equivalence.
Rationally, the problem was solved by Panov 
and Ray~\cite{PR}: every quasitoric manifold is \emph{rationally formal}. This result implies that if two quasitoric manifolds $M$ and $N$ have isomorphic cohomology rings, then there is a rational homotopy equivalence $M\overset{\simeq}{\longrightarrow}N$.
On the other hand, a counterexample of quasitoric manifolds with isomorphic integral cohomology rings but different homotopy types has been found~\cite{GL}. It suggests that cohomological rigidity for quasitoric manifolds is a rather special phenomenon, even at the level of homotopy equivalence.

A natural question is therefore whether a rational homotopy equivalence 
can be upgraded to an integral or a local equivalence after inverting finitely many primes.
In~\cite{FSST} the authors gave a positive answer in the special case when 
$M$ and $N$ are $2n$-dimensional quasitoric manifolds associated to the polytope 
$P=\prod^{\ell}_{i=1}\Delta^d$, a product of simplices of the same dimension $d$. In this case, $n=\ell d$ and it was shown that there is a homotopy equivalence $M\simeq N$ after inverting 
all primes $p\leq n-d+1$. This paper uses different methods to improve this result to the case 
when $P$ is any product of simplices.

\begin{theorem}
   \label{main}
   Let $N$ be a $2n$-dimensional quasitoric manifold associated to $P=\prod_{i=1}^{\ell}\Delta^{d_{i}}$ with each $d_{i}\geq 1$. If $M$ is a quasitoric manifold with a graded ring isomorphism $H^{\ast}(M;\Z)\cong H^{\ast}(N;\Z)$, then there is a homotopy equivalence $M\simeq N$ after localizing away from all primes $p\leq n-d$, where~$d=\min\{d_{1},\ldots, d_{\ell}\}$. 
\end{theorem} 

We note that Theorem \ref{main} admits a refinement (Theorem~\ref{induction}) which establishes a homotopy equivalence $M\simeq N$ after localizing away from a certain finite set of primes. In this more general setting,~$P$ may be any simple polytope satisfying $\pi_{2i}(\mathcal{Z}_P)\otimes\Q=0$ for $i\leq n-1$. In particular, the set of primes to be localized away only depends on $P$. When $P=\prod_{i=1}^{\ell}\Delta^{d_{i}}$ as in Theorem \ref{main},  the list of primes can be made explicit. 

\section*{Acknowledgement}
X.\,Fu is supported by the National Natural Science Foundation of China (grant no. 12501083) and the startup fund of SIMIS. T. So and J. Song are supported by the National Research Foundation of Korea(NRF) grant funded by the Korea government(MSIT) (RS-2025-00555914).

T. So thanks Martin Frankland for discussions on the localization of CW-complexes.

\section{Preliminary information on quasitoric manifolds} 
\label{sec:prelim} 

This section records preliminary facts about quasitoric manifolds that will be used subsequently. 
Throughout, homology and cohomology are taken with integral coefficients unless otherwise indicated.
Let $P$ be an $n$-dimensional simple polytope with $m$ facets. If $M$ 
is a quasitoric manifold associated with $P$ then it is simply-connected, has dimension $2n$, 
and there is a homotopy fibration sequence 
\begin{equation} 
  \label{ZkMfib}
\Omega M\longrightarrow T^{m-n}\overset{\partial}{\longrightarrow}\mathcal{Z}_{P}\longrightarrow{M}\longrightarrow BT^{m-n}
\end{equation}
where $\partial$ is the fibration connecting map.

\begin{lemma} 
   \label{loopMsplit}
   The map $\partial$ is null homotopic.
\end{lemma} 

\begin{proof} 
The space $\mathcal{Z}_{P}$ is $2$-connected~\cite[Proposition 4.3.5]{BP}. Therefore the map 
\(\namedright{M}{}{BT^{m-n}}\) 
induces an isomorphism on $\pi_{2}$, implying that the loop of this map 
\(\namedright{\Omega M}{}{T^{m-n}}\) 
induces an isomorphism on $\pi_{1}$. Each generator in $\pi_{1}(\Omega M)$ is represented 
by a map 
\(\namedright{S^{1}}{}{\Omega M}\), 
and the loop space multiplication allows these to be multiplied together to obtain a map 
\(\namedright{T^{m-n}}{}{\Omega M}\) 
that induces an isomorphism on $\pi_{1}$. Hence the composite 
\(\nameddright{T^{m-n}}{}{\Omega M}{}{T^{m-n}}\) 
induces an isomorphism on $\pi_{1}$ and so is a homotopy equivalence since $T^{m-n}$ is 
an Eilenberg-Mac Lane space. This retraction of $T^{m-n}$ off $\Omega M$ implies that $\partial$ 
is null homotopic. 
\end{proof} 

The cohomology of a quasitoric manifold $M$ was calculated in~\cite{DJ}. They showed that $H^\ast(M)$ is torsion free and there is an isomorphism of graded rings
\[
H^{\ast}(M)\cong\mathbb{Z}[x_{1},\ldots,x_{m}]/\left(\mathcal I+\mathcal J\right)
\] 
where each $x_{i}$ has degree $2$ and $\mathcal{I}$ and $\mathcal{J}$ are certain 
explicitly described ideals. 
In particular,~$H^{\ast}(M)$ 
is generated by elements in degree $2$, so that it is concentrated in even degrees. 
The universal coefficient theorem 
then implies that cohomology is dual to homology, so there is an isomorphism~
$H_{i}(M)\cong H^{i}(M)$ for all $i$. In particular, the homology of $M$ is torsion free and concentrated in even degrees as well. 


We draw two consequences from this homological and cohomological information. 
First, as $M$ has dimension $2n$, its homology
is concentrated in even degrees, from $2$ to $2n$. As $M$ is 
simply-connected and has torsion-free homology, it has a $CW$-structure with one cell corresponding to each homology generator.
Fix 
such a $CW$-structure. For an integer $k\leq n$, let $M(2k)$ be the $2k$-skeleton of $M$. 
Note that $M(0)=\ast$ and $M(2n)=M$. For $0\leq k<n$ there is a homotopy cofibration 
\begin{equation} 
  \label{skelcofib} 
\nameddright{\bigvee_{s=1}^{t_{k}} S^{2k+1}}{f_{k}}{M(2k)}{j_{k}}{M(2k+2)}  
\end{equation}  
where $t_{k}$ is the number of $(2k+2)$-cells in the $CW$-structure for $M$, $f_{k}$ is the attaching 
map for the $(2k+2)$-skeleton and $j_{k}$ is the skeletal inclusion. The following lemma gives 
some information about the homotopy groups of these skeletons.

\begin{lemma} 
   \label{skelconnect} 
   Let $X$ be any $CW$-complex of dimension $d$ where $d\leq 2k+2$. 
   Then the skeletal inclusion  
   \(\namedright{M(2k+2)}{}{M}\) 
   induces a map 
   \(\namedright{[X,{M(2k+2)}]}{}{[X,M]}\) with trivial kernel. 
\end{lemma}

\begin{proof} 
Define the space $F$ by the homotopy fibration 
\[\nameddright{F}{}{M(2k+2)}{}{M}\]  
where the right map is the skeletal inclusion. This induces an exact sequence of homotopy sets 
\begin{equation} 
   \label{Fhtpy} 
   \nameddright{[X,F]}{}{[X,{M(2k+2)}]}{}{[X,M]}.  
\end{equation}  
Since $M$ has homology only in even degrees, $M(2k+2)$ is in fact the $(2k+3)$-skeleton 
of~$M$, implying that $F$ is $(2k+2)$-connected. Since $X$ is $d$-dimensional 
for $d\leq 2k+2$, we obtain $[X,F]\cong 0$ so~
\(\namedright{[X,{M(2k+2)}]}{}{[X,M]}\) 
has a trivial kernel. 
\end{proof} 

Next, there is a result about $2$-skeletons that will be important.

\begin{proposition} \label{skel2rigidity} 
For two quasitoric manifolds $M$ and $N$,
suppose that there is a graded ring isomorphism 
$\gamma \colon H^\ast(N)\to H^\ast(M)$.
Then there is a homotopy equivalence  
$\phi\colon M(2)\to N(2)$
such that $\phi^{\ast}$ equals $\gamma$ on degree~$2$ generators.  
\end{proposition}  

\begin{proof} 
Since there are additive isomorphisms $H_{\ast}(M)\cong H^{\ast}(M)$ and $H_{\ast}(N)\cong H^{\ast}(N)$, the multiplicative isomorphism 
$H^{\ast}(M)\cong H^{\ast}(N)$ implies that there is an isomorphism $H_{\ast}(M)\cong H_{\ast}(N)$ of graded abelian groups. In particular, $M$ and $N$ are simply-connected and both have $2$-skeleton homotopy equivalent to $\bigvee_{i=1}^{s} S^{2}$ for some number $s$.
The Hurewicz Isomorphism implies that any isomorphism of $H_{\ast}(\bigvee_{i=1}^{s} S^{2})$ 
can be realized as a self-homotopy equivalence of $\bigvee_{i=1}^{s} S^{2}$. Thus in 
our case the isomorphism $H_{\ast}(M(2))\cong H_{\ast}(N(2))$ induced 
by the dual of $\gamma$ implies that there is a homotopy equivalence  
 \(\phi\colon\namedright{M(2)}{}{N(2)}\) 
such that $\phi_{\ast}$ equals the dual of $\gamma$ on degree~$2$ generators. Dualizing again, 
$\phi^{\ast}$ equals $\gamma$ on degree~$2$ generators. 
\end{proof} 

\section{$\mathcal{P}$-localization and $CW_{\mathcal{P}}$ complexes}
\label{sec:Plocalization}

Let $\mathcal{P}$ be a set of primes. The  ring of $\mathcal{P}$-local integers is
\[
\Z_{\mathcal{P}}=\left\{\frac{a}{b}\in\Q\mid \forall p\in\mathcal{P},\  \gcd(b,p)=1\right\}.
\]
A simply-connected space $X$ is \emph{$\mathcal{P}$-local} if one of the following equivalent conditions hold:
\begin{itemize}
\item
$\pi_*(X)\to\pi_*(X)\otimes_{\Z}\Z_{\mathcal{P}}$ is an isomorphism;

\item
$\tilde{H}_*(X;\Z)\to\tilde{H}_*(X;\Z)\otimes_{\Z}\Z_{\mathcal{P}}$ is an isomorphism. 
\end{itemize}
A map $f\colon X\to Y$ between simply-connected spaces is a \emph{$\mathcal{P}$-localization} if $Y$ is $\mathcal{P}$-local and
\[
f_*\otimes\Z_{\mathcal{P}}\colon\pi_*(X)\otimes\Z_{\mathcal{P}}\to\pi_*(Y)\otimes\Z_{\mathcal{P}}\cong\pi_*(Y)
\]
is an isomorphism.
When $\mathcal{P}$ is empty, $\Z_{\mathcal{P}}=\Q$ and a $\mathcal{P}$-localization is called a \emph{rationalization}.

A simply-connected CW complex admits a $\mathcal{P}$-localization by attaching additional cells, called the \emph{cellular localization}. This construction has the advantage of the resulting $\mathcal{P}$-local space being a $CW_{\mathcal{P}}$ complex in the
sense of~\cite[Chapter~9]{FHT}; see also \cite[Definition~1.3]{Hess} for the case of rationalization. 
 This is defined as follows.

Let $\{k_1,k_2,\ldots\}$ be the set of positive integers coprime to all primes in $\mathcal{P}$.
For any $n\geq1$, the~\emph{$\mathcal{P}$-local $n$-sphere} is the infinite telescope
\[
S^n_{\mathcal{P}}=\mbox{hocolim}(S^n\overset{k_1}{\longrightarrow}S^n\overset{k_2}{\longrightarrow}\cdots),
\]
where $S^n\xrightarrow{k_i} S^n$ is the degree-$k_i$ map, and the \emph{$\mathcal{P}$-local $(n+1)$-disk} is its cone
\[
D^{n+1}_{\mathcal{P}}=S^n_{\mathcal{P}}\times I/S^n_{\mathcal{P}}\times\{0\}.
\]
A \emph{$CW_\mathcal{P}$ complex} $X$ is a space with a filtration $X=\bigcup^{\infty}_{i=0}X(i)$ such that
\begin{itemize}
\item
$X(0)=pt$,
\item for $n\geq 0$,
$X(n+1)=X(n)\cup_{f_n}(\bigsqcup_{\alpha}D^{n+1}_{\mathcal{P},\alpha})$ where $f_n\colon\bigsqcup_{\alpha}S^n_{\mathcal{P},\alpha}\to X(n)$ is a cellular map.
\end{itemize}

One can show that $CW_{\mathcal{P}}$ complexes are $\mathcal{P}$-local.
The following lemma asserts that a simply-connected CW complex admits a cellular localization and summarizes the properties needed for our later arguments.

\begin{lemma}[\cite{FHT}, Chapter 9b]\label{lem:CW0 filtration} 
Let $X$ be a simply-connected CW complex. Then there is a relative CW-complex $(X_{\mathcal{P}},X)$ with no zero-cells and no one-cells such that:
\begin{enumerate}
\item $X_{\mathcal{P}}=\bigcup^{\infty}_{i=0}X_{\mathcal{P}}(i)$ is a $CW_{\mathcal{P}}$ complex and the inclusion $l\colon X\to X_{\mathcal{P}}$ is a $\mathcal{P}$-localization;
\item the inclusion $X(k)\longrightarrow X_{\mathcal{P}}(k)$ is a $\mathcal{P}$-localization for $ k\geq1$;
\item if $X$ has no $(k+1)$-cells, then
$X_{\mathcal{P}}(k+1)=X_{\mathcal{P}}(k)$;
\item the $k$-skeleton of $X_{\mathcal{P}}$ is a subcomplex of $X_{\mathcal{P}}(k)$, and $X_{\mathcal{P}}(k)$ is a subcomplex of the $(k+1)$-skeleton of $X_{\mathcal{P}}$.
\end{enumerate}
Moreover, for a $\mathcal{P}$-local simply-connected space $Y$, any map $f\colon X\to Y$ extends to a map $g\colon X_{\mathcal{P}}\to Y$ whose homotopy class is determined by the homotopy class of $f$.
\qed
\end{lemma}

Note that cellular localization is compatible with the inclusion relation of the prime sets.
Let $\mathcal{P}$ and $\mathcal{P}'$ be sets of primes such that $\mathcal{P}\subset\mathcal{P}'$.
Then $\Z_{\mathcal{P}'}$ is a subring of $\Z_{\mathcal{P}}$.
For each $k\geq 1$, $X_{\mathcal{P}'}(k)$ is a subcomplex of $X_{\mathcal{P}}(k)$ and the inclusion $\varphi_k\colon X_{\mathcal{P}'}(k)\to X_{\mathcal{P}}(k)$
is a $\mathcal{P}$-localization. 
Moreover, any map $f\colon X_{\mathcal{P}'}(k)\to Y_{\mathcal{P}'}(k)$ induces a map $f_{\mathcal{P}}\colon X_{\mathcal{P}}(k)\to Y_{\mathcal{P}}(k)$ satisfying
a homotopy commutative diagram
\[
\xymatrix{
X_{\mathcal{P}'}(k)\ar[r]^-{f}\ar[d]_-{\varphi_k}	&Y_{\mathcal{P}'}(k)\ar[d]^-{\varphi_k}\\
X_{\mathcal{P}}(k)\ar[r]^-{f_{\mathcal{P}}}		&Y_{\mathcal{P}}(k).
}
\]
The homotopy class of $f_{\mathcal{P}}$ is uniquely determined by the homotopy class of $f$.

When $\mathcal{P}=\emptyset$, we denote the cellular rationalization $X_{\mathcal{P}}$ and its filtration $X_{\mathcal{P}}(k)$ by $X_0$ and $X_0(k)$ respectively.
In the rest of the paper, all $\mathcal{P}$-localizations and rationalizations are assumed to be the cellular localizations in Lemma~\ref{lem:CW0 filtration} unless otherwise stated.

\section{Rational information for quasitoric manifolds} 
\label{sec:rational} 

This section establishes some properties about the rational homotopy theory of quasitoric 
manifolds that will be important in the next section. A standard reference for rational homotopy theory is~\cite{FHT}. A commutative differential 
graded algebra is abbreviated to CDGA. A map of CDGAs that induces an isomorphism in cohomology is called a quasi-isomorphism and is denoted by $\simeq$. 
Let~$(A,d)$ be a CDGA over the ground field $\mathbb{Q}$. If 
$H^{0}(A,d)\cong\mathbb{Q}$ and $H^{1}(A,d)=0$
then there is a minimal CDGA $(\Lambda V,d)$ and a quasi-isomorphism 
\(\namedright{(\Lambda V,d)}{\simeq}{(A,d)}\). 
This CDGA together with the quasi-isomorphism is called a minimal model for $(A,d)$, 
and it is unique up to CDGA isomorphism. The minimal model of a simply-connected space $X$ is the minimal model of $A_{PL}(X)$, where $A_{PL}(-)$ is Sullivan's functor from the category 
of simply-connected spaces and continuous maps to the category of CDGAs. 

Following~\cite[Definition 2.85]{FOT}, a simply-connected space $X$ with minimal model 
$(\Lambda V,d)$ is \emph{rationally formal} if there is a quasi-isomorphism 
\[\namedright{(\Lambda V,d)}{\simeq}{(H^{\ast}(X;\mathbb{Q}),0)},\] 
where $(H^{\ast}(X;\mathbb{Q}),0)$ is the CDGA whose underlying commutative graded algebra is $H^{\ast}(X;\mathbb{Q})$ and whose differential is identically zero. In what follows, this CDGA will be abbreviated to $H^{\ast}(X;\mathbb{Q})$. 
The relevant example in our case is that quasitoric manifolds are rationally formal, as 
proved by Panov and Ray~\cite[Corollary 7.2]{PR}. This is strengthened as follows.

\begin{proposition}\label{lemma realizing Gamma} 
Let $M$ and $N$ be quasitoric manifolds and let $l_M\colon M\to M_0$ and $l_N\colon N\to N_0$ be their rationalizations. For any morphism of graded algebras $\gamma\colon\namedright{\cohlgy{N;\mathbb{Q}}}{}{\cohlgy{M;\mathbb{Q}}}$, there is a map $e\colon M_0\to N_0$  satisfying a commutative diagram

\[
\xymatrix{
H^*(N_0;\Q)\ar[r]^-{e^*}\ar[d]_-{l_N^*}	&H^*(M_0;\Q)\ar[d]^-{l_M^*}\\
H^*(N;\Q)\ar[r]^-{\gamma}		&H^*(M;\Q).
}
\]
\end{proposition} 

\begin{proof} 
The proof proceeds in two steps.

\noindent 
\textit{Step 1}: 
We prove that there exist rationalizations $h_M\colon M\to M_{\Q}$ and $h_N\colon N\to N_{\Q}$, and a map~$\tilde{e}\colon M_{\Q}\to N_{\Q}$ 
satisfying a commutative diagram 
\begin{equation} 
\label{hrdgrm} 
\xymatrix{
H^*(N_{\Q};\Q)\ar[r]^-{\tilde{e}^*}\ar[d]_-{h_N^*}	&H^*(M_{\Q};\Q)\ar[d]^-{h_M^*}\\
H^*(N;\Q)\ar[r]^-{\gamma}					&H^*(M;\Q).
}
\end{equation} 
Here $M_{\Q}$ and $N_{\Q}$ may not be rational CW complexes, and $h_M$ and $h_N$ are not necessarily the cellular localizations described in Lemma~\ref{lem:CW0 filtration}.

Let $\theta_N\colon(\Lambda V_N,d)\to A_{PL}(N)$ be a minimal Sullivan model. The rational formality of $N$ implies that there is a zig-zag of quasi-isomorphisms
\[
H^*(N)\overset{\epsilon_N}{\longleftarrow}(\Lambda V_N,d)\overset{\theta_N}{\longrightarrow}A_{PL}(N).
\]
Recall that the Sullivan realization of $(\Lambda V_N,d)$ is the simplicial set $\langle\Lambda V_N,d\rangle_{\bullet}$ given by
\[
\langle\Lambda V_N,d)\rangle_i=\text{Hom}_{DGA}((\Lambda V_N,d), (A_{PL})_i),
\] 
where $DGA$ is the category of differential graded algebras and $A_{PL}$ is the polynomial differential form.
Since $(\Lambda V_N,d)$ is the minimal model of $A_{PL}(N)$, by~\cite[Theorem 17.12]{FHT}, the geometric realization $|\Lambda V_N,d|=|\langle \Lambda V_N,d\rangle_{\bullet}|$ gives rise to a rationalization of $N$
\[
h_N\colon N\overset{\simeq}{\longrightarrow}|Sing_{\bullet}({N})|\longrightarrow|\Lambda V_N,d|.
\] 
Similarly, take a minimal model $(\Lambda V_M,d)$ of $M$ and construct the rationalization $h_M\colon M\to|\Lambda V_M,d|$.

We construct the map $\tilde{e}\colon|\Lambda V_M,d|\to|\Lambda V_N,d|$ as follows.
Let $\psi_N\colon A_{PL}(N)\to C^*(N)$ and $\psi_M\colon A_{PL}(M)\to C^*(M)$ be the natural quasi-isomorphisms given in~\cite[Theorem 10.15]{FHT},
and for any morphism of graded $\Q$-algebras $\gamma\colon H^*(N;\Q)\to H^*(M;\Q)$, let $\gamma'$ be the composite
\[
H^*(N;\Q) 
\xrightarrow{\psi_N^*\circ\theta^*_N\circ(\epsilon^*_N)^{-1}}
H^*(N;\Q)\overset{\gamma}{\longrightarrow}H^*(M;\Q) 
\xrightarrow{\epsilon^*_M\circ(\theta^*_M)^{-1}\circ(\psi^*_M)^{-1}}
H^*(M;\Q).
\]
By~\cite[Theorem 12.9]{FHT} there is a lift $\varphi$ of $\gamma'\circ\epsilon_N$ through $\epsilon_M$
\begin{equation}\label{eq_diag_M_H}
\xymatrix{  
       (\Lambda V_N,d)\ar@{-->}[r]^-{\varphi}\ar[d]^{\epsilon_{N}} & (\Lambda V_M,d)\ar[d]^{\epsilon_{M}} \\ 
       H^{\ast}(N;\Q)\ar[r]^-{\gamma'}					&H^{\ast}(M;\Q).
  }
\end{equation}
Take $\tilde{e}=|\langle\varphi\rangle|$.
We need to show that it satisfies the commutative diagram~(\ref{hrdgrm}).

Consider the diagram of DGAs
\begin{equation}
  \spreaddiagramcolumns{1pc}
  \diagram 
  & H^*(N;\Q)\rto^-{\gamma'} & H^*(M;\Q)&\\
 &   (\Lambda V_N,d)\rto^-{\varphi}\dto^{\eta_{N}}
 \ar@/_0.7pc/[dl]_-{\theta_N}
 \uto^{\epsilon_N} & (\Lambda V_M,d)\uto_-{\epsilon_M}\dto^{\eta_{M}}\ar@/^0.7pc/[dr]^-{\theta_M}
 \\ 
 A_{PL}(N)\dto_-{\psi_N} &     A_{PL}(\vert\Lambda V_N,d\vert)\rto^-{A_{PL}(\tilde{e})} \lto_-{A_{PL}(h_N)}
        \dto  & A_{PL}(\vert\Lambda V_M,d\vert)\dto\rto^-{A_{PL}(h_M)}
  &A_{PL}(M) \dto^-{\psi_M}     
        \\
C^*(N) &    C^*(|\Lambda V_N,d|)\rto_-{C^*(\tilde{e})}\lto^-{C^*(h_N)} &C^*(|\Lambda V_M,d|)\rto_-{C^*(h_M)} & C^*(M).
 \enddiagram 
\end{equation} 
where $\eta_N,\eta_M$ are the quasi-isomorphisms given in~\cite[Theorem 17.10]{FHT}.
The top square is Diagram~\eqref{eq_diag_M_H}, the two triangles are homotopy commutative
due to~\cite[Theorem 17.12]{FHT}, the middle square is homotopy commutative
by~\cite[Diagram (17.14)]{FHT}, and the remaining squares commute due to functoriality. 
Taking cohomology, the outer perimeter implies that there is a commutative diagram 
\[
\xymatrix{
H^*(N;\Q)\ar[rrrr]^-{\psi^*_M\circ\theta_M^*\circ(\epsilon^*_M)^{-1}\circ\gamma'\circ\epsilon^*_N\circ(\theta^*_N)^{-1}\circ(\psi^*_N)^{-1}}	&&&&H^*(M;\Q)\\
H^*(|\Lambda V_N,d|;\Q)\ar[rrrr]^-{\tilde{e}^*}\ar[u]^-{h_N^*}							&&&&H^*(|\Lambda V_M,d|;\Q)\ar[u]_-{h_M^*}.
}
\]
Observe that the top row equals $\gamma$ by definition of~$\gamma'$.
This proves the existence of~(\ref{hrdgrm}).
\smallskip 

\noindent 
\textit{Step 2}:
We prove the commutativity of the diagram asserted in the statement of the proposition. 
The universal property of
rationalization yields homotopy equivalences
\[
g_M\colon |\Lambda V_M,d|\longrightarrow M_{0}
\quad\text{and}\quad
g_N\colon |\Lambda V_N,d|\longrightarrow N_0
\]
such that $l_M\simeq g_M\circ h_M$ and $l_N\simeq g_N\circ h_N$.
Define $e\colon M_0\to N_0$ as the composite
\[
M_0\overset{g_M^{-1}}{\longrightarrow}|\Lambda V_M,d|\overset{\tilde{e}}{\longrightarrow}|\Lambda V_N,d|\overset{g_N}{\longrightarrow}N_0.
\] 
Then there is homotopy commutative diagram
\[
\xymatrix{
M\ar[r]^-{h_M}\ar@{=}[d]   &|\Lambda V_M,d|\ar[r]^-{\tilde{e}}\ar[d]^-{g_M}   &|\Lambda V_N,d|\ar[d]\ar[d]^-{g_N}   &N\ar[l]_-{h_N}\ar@{=}[d]\\
M\ar[r]^-{l_M}   &M_0\ar[r]^-{e}   &N_0   &N\ar[l]_-{l_N}
}
\]
The zig-zags from $M$ to $N$ in the top and the bottom rows induce the same cohomology homomorphisms. Using
(\ref{hrdgrm}) we obtain
$l^*_M\circ e^*\circ (l^*_N)^{-1}=h_M^*\circ\tilde{e}^*\circ(h_N^*)^{-1}=\gamma$, 
as required.
\end{proof}

Let $\mathcal{P}$ be a set of primes. Apply the $\mathcal{P}$-localization $M_{\mathcal{P}}=\bigcup^{\infty}_{i=0}M_{\mathcal{P}}(i)$ in Lemma~\ref{lem:CW0 filtration} to a quasitoric manifold $M$. Since $M$ is $2n$-dimensional and consists of even-dimensional cells only, for each $k\geq 1$ we have $M_{\mathcal{P}}(2k+1)=M_{\mathcal{P}}(2k)$ and $M_{\mathcal{P}}(2n)=M_{\mathcal{P}}$.
The same is true for the~$\mathcal{P}$-localization~$N_{\mathcal{P}}=\bigcup^{\infty}_{i=0}N_{\mathcal{P}}(i)$ of $N$. 

\begin{lemma}\label{lem:restriction}
Suppose that $e\colon M_0\longrightarrow N_0$
is a homotopy equivalence.
Then there exist maps
\begin{equation}\label{eq_e_2k}
e_{2k}\colon M_0(2k)\longrightarrow N_0(2k),
\qquad
1\le k\le n,
\end{equation}
fitting into a homotopy commutative diagram
\begin{equation} 
\label{eladder 2}\spreaddiagramcolumns{-0.5pc}  
\diagram 
M_0(2)\rto\dto^{e_{2}} & M_{0}(4)\rto\dto^{e_4} & \cdots\rto & M_{0}(2k-2)\rto\dto^{ e_{2k-2}}
& M_{0}(2k)\rto\dto^{e_{2k}} & \cdots\rto & M_0\dto^{e} \\ 
N_0(2)\rto & N_{0}(4)\rto & \cdots\rto & N_{0}(2k-2)\rto & N_{0}(2k)\rto & \cdots\rto & N_0. 
  \enddiagram 
\end{equation}
Moreover, each $e_{2k}$ is a homotopy equivalence.
\end{lemma}

\begin{proof}
Let $\sk_i(M_0)$ denote the $i$-skeleton of $M_0$.
Applying the Cellular Approximation Theorem to the map
$e\colon M_0\longrightarrow N_0$
gives a homotopy commutative diagram
\begin{equation} 
\label{eladder}\spreaddiagramcolumns{-0.5pc}  
\diagram 
\sk_{2}(M_0)\rto\dto^{ \tilde e_{2}} & \sk_{3}(M_{0})\rto\dto^{\tilde e_{3}} & \cdots\rto & \sk_{k}(M_{0})\rto\dto^{\tilde e_{k}} 
& \sk_{k+1}(M_{0})\rto\dto^{\tilde e_{k+1}} & \cdots\rto & M_0\dto^{e} \\ 
\sk_{2}(N_0)\rto & \sk_{3}(N_{0})\rto & \cdots\rto & \sk_{k}(N_{0})\rto & \sk_{k+1}(N_{0})\rto & \cdots\rto & N_0
  \enddiagram 
\end{equation}  
whose skeletal restrictions define the maps $\tilde{e}_{k}$.

By Lemma~\ref{lem:CW0 filtration}, we have $M_0(2k)\subset \sk_{2k+1}(M_0)$ and $\sk_{2k+1}(N_0)\subset N_0(2k+1)$. Hence the restriction of $\tilde{e}_{2k+1}$ to $M_{0}(2k)$ defines a map
\[e_{2k}\colon M_0(2k) \to N_0(2k+1).\]
Since $N$ has no $(2k+1)$-cells, we have $N_0(2k+1)=N_0(2k)$. Hence $e_{2k}$ is a map $M_0(2k) \to N_0(2k)$ for all $k\geq 1$.
Restricting Diagram~\eqref{eladder} in this way yields the
homotopy commutative diagram~\eqref{eladder 2}.

It remains to prove that each $e_{2k}$ is a homotopy equivalence. 
For each $k$, the square
\begin{equation}\label{eq_htpy_square_inclusion}
\begin{tikzcd}
M_0(2k)
\arrow[r,hook]
\arrow[d,"e_{2k}"']
&
M_0
\arrow[d,"e"]
\\
N_0(2k)
\arrow[r,hook]
&
N_0
\end{tikzcd}
\end{equation}
is homotopy commutative.
By Lemma~\ref{lem:CW0 filtration}, the spaces $M_0(2k)$ and $N_0(2k)$
are rationalizations of~$\sk_{2k} (M)$ and $\sk_{2k} (N)$, respectively.
Since $M$ and $N$ have no $(2k+1)$-cells, inclusions
\[\text{$M_0(2k)\to M_0$ and $N_0(2k)\to N_0$}
\]
induce isomorphisms on rational homology in every degree $\leq 2k$.
Since $e$ is a rational homotopy equivalence, it induces an isomorphism on rational homology. The homotopy commutativity of \eqref{eq_htpy_square_inclusion} therefore
implies that
\[
(e_{2k})_*\colon
H_i\bigl(M_0(2k);\mathbb Q\bigr)
\longrightarrow
H_i\bigl(N_0(2k);\mathbb Q\bigr)
\]
is an isomorphism for every $i\leq 2k$.

Finally, since $M_0(2k)$ and $N_0(2k)$ have rational cells only in dimensions at most $2k$, their rational homology vanishes in degrees greater than $2k$. Hence $(e_{2k})_*$ is an isomorphism in every degree. As~$M_0(2k)$ and $N_0(2k)$ are simply-connected, Whitehead's Theorem implies that $e_{2k}$ is a rational homotopy equivalence. Since $M_0(2k)$ and $N_0(2k)$ are rational spaces, $e_{2k}$ itself is a homotopy equivalence.
\end{proof}

By Lemma~\ref{lem:restriction} there is a collection of rational equivalences between filtrations of $M_0$ and $N_0$.
We ask whether there exists a collection of $\mathcal{P}$-local equivalences $M_{\mathcal{P}}(k)\to N_{\mathcal{P}}(k)$. The following lemma gives a necessary condition. 

Recall from the end of Section~\ref{sec:Plocalization} that cellular localization is compatible with the inclusion relation of prime sets. In particular, if $\mathcal{P}$ is a set of primes then $M_{\mathcal{P}}(2k)$ is a subcomplex of $M_{0}(2k)$ and the inclusion 
\(\lnamedright{M_{\mathcal{P}}(2k)}{\varphi_{2k}}{M_{0}(2k)}\) 
is a rationalization, and similarly for 
\(\lnamedright{N_{\mathcal{P}}(2k)}{\varphi_{2k}}{N_{0}(2k)}\).

\begin{lemma}\label{rationalextension} 
Let $\mathcal{P}$ be a set of primes and fix an integer $k$ such that $1\leq k< n$. Suppose that~\(\namedright{M_{\mathcal{P}}(2k)}{h}{N_{\mathcal{P}}(2k)}\)   
   is a map such that there is a homotopy $h_0\simeq e_{2k}$ for $e_{2k}$ as in Lemma~\ref{lem:restriction}, and that the polytope $P$ associated to $N$ satisfies~$\pi_{2k+2}(\mathcal{Z}_P)\otimes\Q=0$.
   If the composite 
   \[
   M_{\mathcal{P}}(2k)\overset{h}{\longrightarrow}
   N_{\mathcal{P}}(2k)\overset{j_{k}(N)}{\longrightarrow}N_{\mathcal{P}}(2k+2)\overset{\varphi_{2k+2}}{\longrightarrow}N_{0}(2k+2)
   \] 
   extends across $\varphi_{2k+2}\circ j_{k}(M)$ to give a homotopy commutative diagram
   \[\diagram 
         M_{\mathcal{P}}(2k)\rto^-{j_{k}(M)}\dto^{h} &M_{\mathcal{P}}(2k+2)\rto^-{\varphi_{2k+2}} & M_0(2k+2)\ar@{-->}[d]^{h'} \\    
         N_{\mathcal{P}}(2k)\rto^-{j_{k}(N)} &N_{\mathcal{P}}(2k+2)\rto^-{\varphi_{2k+2}}	&N_0(2k+2)   
      \enddiagram\]
      for some map $h'$, then there is a homotopy $h'\simeq e_{2k+2}$. 
\end{lemma}

\begin{proof} 
Mapping the homotopy cofibration sequence  
\[\namedddright{\bigvee_{s=1}^{t_{k}} S^{2k+1}_{\mathcal{P}}}{f_{k}}{M_{\mathcal{P}}(2k)}{j_{k}(M)}{M_{\mathcal{P}}(2k+2)}{\delta_{k}}{\bigvee_{s=1}^{t_{k}} S^{2k+2}_{\mathcal{P}}}\]  
to $N_{0}(2k+2)$ gives an exact sequence of pointed sets 
\begin{multline*}
     \big[ \bigvee_{s=1}^{t_{k}} S^{2k+2}_{\mathcal{P}},N_{0}(2k+2)\big] \xrightarrow{\delta_k^\ast} [M_{\mathcal{P}}(2k+2),N_{0}(2k+2)] \\
      \xrightarrow{j_k(M)^\ast} [M_{\mathcal{P}}(2k),N_{0}(2k+2)] \xrightarrow{f_k^\ast} \big[\bigvee_{s=1}^{t_{k}} S^{2k+1}_{\mathcal{P}},N_{0}(2k+2)\big].
\end{multline*}

First, consider $[\bigvee_{s=1}^{t_{k}} S^{2k+2}_{\mathcal{P}},N_{0}(2k+2)]$. 
By~\eqref{ZkMfib}, there is a homotopy fibration $\mathcal{Z}_P\to N\to BT^{\ell}$. Rationally, this gives a homotopy fibration  
$(\mathcal{Z}_{P})_{0}\to N_0\to BT^{\ell}_0$. 
As $BT^{\ell}_0$ is the Eilenberg-Mac Lane space $K(\mathbb{Q}^{\ell},2)$, for $k\geq 1$ there is an isomorphism 
$\pi_{2k+2}((\mathcal{Z}_{P})_{0})\cong\pi_{2k+2}(N_{0})$. 
By hypothesis,~$\pi_{2k+2}((\mathcal{Z}_{P})_0)\cong 0$, implying 
that $\pi_{2k+2}(N_0)\cong 0$. Further, 
since $N_{0}$ is obtained from $N_{0}(2k+2)$ by attaching cells of dimensions~$\geq 2k+4$
the inclusion $N_{0}(2k+2)\to N_0$ induces an isomorphism~$\pi_{2k+2}(N_{0}(2k+2))\cong \pi_{2k+2}(N_0).$
Thus $\pi_{2k+2}(N_0(2k+2))\cong 0$. Hence 
$[\bigvee_{s=1}^{t_{k}} S^{2k+2}_{\mathcal{P}},N_{0}(2k+2)]\cong 0$, 
implying that $\delta_{k}^{\ast}$ is the zero map. 
Therefore $j_k(M)^*$ is injective by \cite[Theorem~4.4.8]{arkowitz}. 

Next, the existence of an extension of $\varphi_{2k+2}\circ j_{k}(N)\circ h$ across $j_{k}(M)$ is 
equivalent to the homotopy class of $\varphi_{2k+2}\circ j_{k}(N)\circ h$ lifting through $j_{k}(M)^{\ast}$. Since 
$j_{k}(M)^{\ast}$ is an injection, if such a lift exists it is unique up to homotopy. Moreover, since $N_{0}(2k+2)$ is a rational space, the universal property of rationalizations implies that any map 
$M_{\mathcal{P}}(2k+2)\stackrel{g}{\longrightarrow} N_{0}(2k+2)$ factors as a composite 
$M_{\mathcal{P}}(2k+2)\stackrel{\varphi_{2k+2}}{\llarrow} M_{0}(2k+2)\stackrel{g'}{\longrightarrow} N_0(2k+2)$, 
where $g'$ is the rationalization of $g$. Hence if $\varphi_{2k+2}\circ j_{k}(N)\circ h$ extends across $j_{k}(M)$, then it extends across $\varphi_{2k+2}\circ j_{k}(M)$, and does so in a unique way up to homotopy. 

Finally, consider the diagram 
\[\diagram 
     M_{\mathcal{P}}(2k)\rto^-{\varphi_{2k}}\dto^{h} &M_{0}(2k)\rto\dto^{e_{2k}} & M_0(2k+2)\dto^{e_{2k+2}} \\    
     N_{\mathcal{P}}(2k)\rto^-{\varphi_{2k}} &N_{0}(2k)\rto	&N_0(2k+2)   
  \enddiagram\]
The left square homotopy commutes since, by hypothesis, $h_{0}\simeq e_{2k}$. The right square homotopy commutes by Lemma~\ref{lem:restriction}. 
Observe that by rationalizing maps in a different order the top row is 
homotopic to the composite 
\(\llnameddright{M_{\mathcal{P}}(2k)}{j_{k}(M)}{M_{\mathcal{P}}(2k+2)}{\varphi_{2k+2}}{M_{0}(2k+2)}\) 
and similarly the bottom row is homotopic to $\varphi_{2k+2}\circ j_{k}(N)$. 
Thus we obtain $e_{2k+2}\circ\varphi_{2k+2}\circ j_{k}(M)\simeq\varphi_{2k+2}\circ j_{k}(N)\circ h$. Therefore, if there is an 
extension $h'$ of $\varphi_{2k+2}\circ j_{k}(N)\circ h$ across $\varphi_{2k+2}\circ j_{k}(M)$ as in the statement of the lemma, then the uniqueness result for such an extension proved in the previous paragraph implies that $h'\simeq e_{2k+2}$. 
\end{proof}

\section{The proof of Theorem~\ref{main}} 
\label{sec:proof} 
The results in this section will be stated for any simple polytope until we reach the proof 
of Theorem~\ref{main}. Let $M$ and $N$ be two quasitoric manifolds.
Suppose that there is an 
isomorphism of integral cohomology rings 
\begin{equation}\label{eq_gamma_integral_iso}
\gamma\colon\namedright{H^{\ast}(N)}{}{H^{\ast}(M)}.
\end{equation}
By Proposition~\ref{lemma realizing Gamma}, there is a homotopy equivalence 
\(e\colon\namedright{M_0}{}{N_0}\), together with cellular rationalizations $l_M\colon M\to M_0$ and $l_N\colon N\to N_0$, such that $l_M^*\circ e^*\circ (l_N^*)^{-1}=\gamma\otimes\mathbb{Q}.$
As in Lemma~\ref{lem:restriction}, for~$1\leq k\leq n$, let 
\(e_{2k}\colon\namedright{M_0(2k)}{}{N_0(2k)}\) 
be the restriction of $e$ to rational $2k$-skeletons. 

We wish to show that the isomorphism $\gamma$ implies a stronger 
result, that there is a local homotopy equivalence $M\simeq N$ after inverting finitely many primes.  
To be more precise, the strategy to prove Theorem~\ref{main} is to show that each homotopy equivalence of rational skeletons $e_{2k}\colon M_{0}(2k)\to N_0(2k)$ can be lifted to a $\mathcal{P}_k$-local equivalence $\Phi_{2k}\colon M_{\mathcal{P}_k}(2k)\to N_{\mathcal{P}_k}(2k)$ by induction on $k$
for some decreasing sequence $\mathcal{P}_1\supset\cdots\supset\mathcal{P}_n$ of sets of primes.
In particular, when $k=n$ we obtain a $\mathcal{P}_n$-local equivalence $\Phi_{2n}\colon M_{\mathcal{P}_n}\to N_{\mathcal{P}_n}$.

To begin, we show that $e_2\colon M_0(2)\to N_0(2)$ is the rationalization of an integral homotopy equivalence between the $2$-skeletons of $M$ and $N$. 

\begin{lemma} 
   \label{skel2QZ} 
There is a homotopy equivalence $\Phi_2\colon M(2)\to N(2)$ such that $(\Phi_2)_0\simeq e_{2}$. 
\end{lemma} 

\begin{proof} 
By Proposition~\ref{skel2rigidity} there exists a homotopy equivalence $\phi\colon M(2)\to N(2)$ such that~$\phi^*$ equals the isomorphism $\gamma$ on degree two cohomology. Let $\phi_0$ be its rationalization. 

With respect to the $CW$-structure on $M$ discussed in Section \ref{sec:prelim}, it follows that $M(2)$ is homotopy equivalent to a finite wedge of $2$-spheres.
Thus $M_0(2)\simeq\bigvee_iS^2_{0}$, implying that it is a co-$H$-space, 
so we may take the difference 
\(\llnamedright{M_0(2)}{\phi_0-e_{2}}{N_0(2)}\). 
We will show that $\phi_0-e_{2}$ is null homotopic.

Since $M_0(2)\simeq\bigvee_iS^2_{0}$, we have $[M_0(2),-]\cong\bigoplus_i\pi_2(-)\otimes\Q$.
Consider the commutative diagram
\[
\xymatrix{
[M_0(2),M_0(2)]\ar[r]^-{\cong}\ar[d]^-{(\phi_0-e_2)_*}	&\bigoplus_i\pi_2(M_0(2))\otimes\Q\ar[r]^-{\cong}\ar[d]^-{(\phi_0-e_2)_*}	&\bigoplus_iH_2(M_0(2);\Q)\ar[d]^-{(\phi_0-e_2)_*}\\
[M_0(2),N_0(2)]\ar[r]^-{\cong}							&\bigoplus_i\pi_2(N_0(2))\otimes\Q\ar[r]^-{\cong}							&\bigoplus_iH_2(N_0(2);\Q)
}
\]
where the horizontal isomorphisms in the right square are rational Hurewicz isomorphisms; see~\cite[Lemma 8.7]{FHT}.
The right vertical map $(\phi_0-e_2)_*$ is trivial since its dual 
\[(\phi_0-e_2)^*\colon H^2(N_0(2);\Q)\to H^2(M_0(2);\Q)\] is trivial because $\phi_0^*$ and $e_2^*$ both equal $\gamma\otimes\mathbb{Q}$ in degree two cohomology. Therefore the left vertical map $(\phi_0-e_2)_*$ is also trivial, implying that $\phi_0-e_2$ is null homotopic. Thus $\phi_{0}\simeq e_{2}$, so taking~$\Phi_{2}=\phi$ proves the lemma.
\end{proof}

Next we prove the inductive case. Let $P$ be the simple polytope associated with $N$. By
~(\ref{ZkMfib}) there is a homotopy fibration  
\(\nameddright{\mathcal{Z}_{P}}{}{N}{}{BT^{m-n}}\), 
implying that $\pi_{\ell}(N)\cong\pi_{\ell}(\mathcal{Z}_{P})$ for all $\ell>2$. This lets us phrase homotopy group hypotheses for $N$ in terms of $\mathcal{Z}_{P}$. As $\mathcal{Z}_{P}$ 
is a simply-connected space of finite type, 
$\pi_{\ell}(\mathcal{Z}_{P})$ is a finitely generated abelian group for $\ell\geq 2$, 
so $\pi_{\ell}(\mathcal{Z}_{P})\cong\pi_{\ell}^{\mbox{\scriptsize{Free}}}(\mathcal{Z}_{P})\oplus\pi_{\ell}^{\mbox{\scriptsize{Tor}}}(\mathcal{Z}_{P})$, where 
$\pi_{\ell}^{\mbox{\scriptsize{Free}}}(\mathcal{Z}_{P})$ is free abelian and $\pi_{\ell}^{\mbox{\scriptsize{Tor}}}(\mathcal{Z}_{P})$ is torsion. 
For any $k\geq 1$, let $\mathcal{P}_k$ be the set of all primes except those appearing in $\pi_{2j+1}^{\mbox{\scriptsize{Tor}}}(\mathcal{Z}_{P})$ for $1\leq j< k$ and let $R_{k}$ be $\Z_{\mathcal{P}_k}$.
We set $\mathcal{P}_1$ to be the set of all primes and $R_1=\Z$ by convention.
We use 
the phrase $R_{k}$-cohomology to refer to cohomology with coefficients in $R_{k}$.


By abuse of notation, we use the same notation for the attaching map $f_k$ (see~\eqref{skelcofib}) induced after~$\mathcal{P}_k$-localization.

\begin{proposition} 
   \label{keyprop}
   Suppose for a fixed $1\leq k<n$ that there is a homotopy equivalence 
   \[\Phi_{2k}\colon\namedright{M_{\mathcal{P}_k}(2k)}{}{N_{\mathcal{P}_k}(2k)}\] 
   such that there is a homotopy $(\Phi_{2k})_{0}\simeq e_{2k}$. Localize away from all primes 
   appearing in $\pi_{2k+1}^{\mbox{\scriptsize{Tor}}}(\mathcal{Z}_{P})$. Then the composite 
   \[\lnamedddright{\bigvee_{s=1}^{t_{k}} S^{2k+1}_{\mathcal{P}_{k+1}}}{f_{k}}{M_{\mathcal{P}_{k+1}}(2k)}{\Phi_{2k}}{N_{\mathcal{P}_{k+1}}(2k)} 
          {j_{k}(N)}{N_{\mathcal{P}_{k+1}}(2k+2)}\] 
   is null homotopic. 
\end{proposition} 

\begin{proof} 
By Lemma~\ref{skelconnect}, the skeletal inclusion 
\(\namedright{N(2k+2)}{}{N}\) 
induces a monomorphism on $\pi_{2k+1}$. 
Thus it suffices to show that the composite
\[\lnamedddright{\bigvee_{s=1}^{t_{k}} S^{2k+1}_{\mathcal{P}_{k+1}}}{f_{k}}{M_{\mathcal{P}_{k+1}}(2k)}{\Phi_{2k}}{N_{\mathcal{P}_{k+1}}(2k)}{I}{N_{\mathcal{P}_{k+1}}}\] 
is null homotopic, where $I$ is the $\mathcal{P}_{k+1}$-localization of the skeletal inclusion. By~(\ref{ZkMfib}) there is a homotopy fibration sequence 
\[\namedddright{T^{m-n}}{\partial}{\mathcal{Z}_{P}}{}{N}{}{BT^{m-n}}\] 
where $\partial$ is the fibration connecting map. 

First, observe that the composite 
\[\lnamedddright{\bigvee_{s=1}^{t_{k}} S^{2k+1}_{\mathcal{P}_{k+1}}}{f_k}{M_{\mathcal{P}_{k+1}}(2k)}{\Phi_{2k}}{N_{\mathcal{P}_{k+1}}(2k)}{I}{N_{\mathcal{P}_{k+1}}}\longrightarrow BT^{m-n}_{\mathcal{P}_{k+1}}\] 
is null homotopic since it represents a class in $\bigoplus_{s=1}^{t_{k}}\pi_{2k+1}(BT^{m-n}_{\mathcal{P}_{k+1}})$, but 
$BT^{m-n}_{\mathcal{P}_{k+1}}$ is the Eilenberg-Mac Lane space $K(\mathbb{Z}^{m-n}_{\mathcal{P}_{k+1}},2)$, implying that $\pi_{2k+1}(BT^{m-n}_{\mathcal{P}_{k+1}})\cong 0$. Therefore, there is a lift 
\begin{equation} 
  \label{lambdalift} 
  \diagram 
     & \displaystyle\bigvee_{s=1}^{t_{k}} S^{2k+1}_{\mathcal{P}_{k+1}}\dto^{I\circ\Phi_{2k}\circ f_{k}}\dlto_{\lambda} \\ 
     (\mathcal{Z}_{P})_{\mathcal{P}_{k+1}}\rto & N_{\mathcal{P}_{k+1}} 
  \enddiagram 
\end{equation} 
for some map $\lambda$. 

Second, observe that $I\circ\Phi_{2k}\circ f_{k}$ 
is rationally null homotopic. For the homotopy commutativity of~\eqref{eladder 2} implies that  $j_{k}(N)\circ\Phi_{2k}$ is rationally homotopic to 
$j_k(N)\circ e_{2k}\simeq e_{2k+2}\circ j_{k}(M)$. As the map~$I$ factors through $j_{k}(N)$, this implies that 
$I\circ\Phi_{2k}$ rationally factors through $e_{2k+2}\circ j_{k}(M)$. Thus the composite $I\circ\Phi_{2k}\circ f_{k}$ 
rationally factors through $e_{2k+2}\circ j_{k}(M)\circ f_{k}$, which is null homotopic since $f_{k}$ 
and $j_{k}(M)$ are consecutive maps in a homotopy cofibration~\eqref{skelcofib}. Hence $I\circ\Phi_{2k}\circ f_{k}$ 
is rationally null homotopic. The homotopy commutativity of~(\ref{lambdalift}) therefore implies that the composite 
\(\nameddright{\bigvee_{s=1}^{t_{k}} S^{2k+1}_{0}}{\lambda}{(\mathcal{Z}_{P})_{0}}{}{N_{0}}\) 
is rationally null homotopic. Thus, rationally, there is a lift 
\[\diagram 
     & \displaystyle\bigvee_{s=1}^{t_{k}} S^{2k+1}_{0}\dto^{\lambda}\dlto &  \\ 
     T^{m-n}_{0}\rto^-{\partial} & {(\mathcal{Z}_{P})_{0}}\rto & N_{0}. 
  \enddiagram\] 

Third, observe that $\lambda$ is null homotopic. 
By Lemma~\ref{loopMsplit}, the fibration connecting map $\partial$ is null homotopic, so the rational lift of $\lambda$ through $\partial$ implies that $\lambda$ is rationally null homotopic. 
By hypothesis, we have localized away from all primes appearing in 
$\pi_{2k+1}^{\mbox{\scriptsize{Tor}}}(\mathcal{Z}_P)$, so $\pi_{2k+1}((\mathcal{Z}_P)_{\mathcal{P}_{k+1}})$ is a $\Z_{\mathcal{P}_{k+1}}$-free module. Hence $\lambda$ is rationally null homotopic if and only if it is null homotopic. 

Finally, the homotopy commutativity of~(\ref{lambdalift}) implies that the composite 
$I\circ\Phi_{2k}\circ f_{k}$ is also null homotopic. 
\end{proof}

\begin{corollary} 
   \label{keypropcor} 
   Suppose that for a fixed $1\leq k<n$ there is a $\mathcal{P}_{k}$-local homotopy equivalence 
   \[\Phi_{2k}\colon\namedright{M_{\mathcal{P}_k}(2k)}{}{N_{\mathcal{P}_k}(2k)}\] 
   such that $\Phi_{2k}^{\ast}$ equals $\gamma$ (see \eqref{eq_gamma_integral_iso}) in $R_{k}$-cohomology in 
   degrees~$\leq 2k$ and 
   $(\Phi_{2k})_0\simeq e_{2k}$. Then after 
   inverting the primes in $\pi_{2k+1}^{\mbox{\scriptsize{Tor}}}(\mathcal{Z}_{P})$ the composite 
   \(\nameddright{M_{\mathcal{P}_{k+1}}(2k)}{\Phi_{2k}}{N_{\mathcal{P}_{k+1}}(2k)}{j_{k}(N)}{N_{\mathcal{P}_{k+1}}(2k+2)}\) 
   extends across $j_{k}(M)$ to give a $\mathcal{P}_{k+1}$-local homotopy commutative diagram 
   \begin{equation} 
     \label{skeldgrm} 
     \diagram 
         M_{\mathcal{P}_{k+1}}(2k)\rto^-{j_{k}(M)}\dto^{\Phi_{2k}} & M_{\mathcal{P}_{k+1}}(2k+2)\dto^{\Phi_{2k+2}} \\ 
         N_{\mathcal{P}_{k+1}}(2k)\rto^-{j_{k}(N)} & N_{\mathcal{P}_{k+1}}(2k+2)  
      \enddiagram 
    \end{equation} 
   where $\Phi_{2k+2}$ is an $\mathcal{P}_{k+1}$-local homotopy equivalence such that $\Phi_{2k+2}^{\ast}$ equals~$\gamma$ in $R_{k+1}$-cohomology in  degrees~$\leq 2k+2$.  
   
   Further, if $\pi_{2k+2}(\mathcal{Z}_P)\otimes\mathbb{Q}=0$ 
   then there is a homotopy $(\Phi_{2k+2})_0\simeq e_{2k+2}$. 
\end{corollary} 

\begin{proof} 
The existence of the extension follows from the null homotopy in Proposition~\ref{keyprop}. 
By hypothesis, $\Phi_{2k}^{\ast}$ equals $\gamma$ in $R_{k}$-cohomology in 
degrees~$\leq 2k$. 
Recall that the integral cohomology ring of $M$ is generated by elements in degree $2$. Hence, each skeleton  of the $CW$ structure discussed in Section~\ref{sec:prelim} on $M$ has the same property, and the same is therefore true in $R_{k}$-cohomology.  
As~$k\geq 1$, this implies that $\Phi_{2k}$ 
has the same image on generators in $R_{k}$-cohomology as $\gamma$ does. The existence 
of the homotopy commutative square~(\ref{skeldgrm}) after further localizing away from the 
primes in $\pi_{2k+1}^{\mbox{\scriptsize{Tor}}}(\mathcal{Z}_{P})$ then implies that $\Phi_{2k+2}^{\ast}$ 
has the same image on generators in $R_{k+1}$-cohomology as $\gamma$. 

Since $\Phi_{2k+2}^{\ast}$  
and $\gamma$ are ring homomorphisms, we obtain 
$\Phi_{2k+2}^{\ast}(x_{i_{1}}\cdots x_{i_{\ell}})=\gamma(x_{i_{1}}\cdots x_{i_{\ell}})$ 
for all monomials $x_{i_{1}}\cdots x_{i_{\ell}}$ of length $\ell\leq k+1$ formed from 
degree~$2$ generators $x_{i_{1}},\ldots,x_{i_{\ell}}$. Since the~$R_{k+1}$-cohomologies of $M$ and $N$ 
are generated by degree $2$ elements, it follows that $\Phi_{2k+2}^{\ast}$ equals~
$\gamma$ in~$R_{k+1}$-cohomology in degrees~$\leq 2k+2$. As $\gamma$ is an isomorphism 
in integral cohomology, it is also an isomorphism in $R_{k+1}$-cohomology. Therefore, 
$\Phi_{2k+2}^{\ast}$ is an isomorphism in $R_{k+1}$-cohomology in degrees~$\leq 2k+2$. 
But as $H^{>2k+2}(M(2k+2);R_{k+1})\cong H^{>2k+2}(M(2k+2);R_{k+1})\cong0$,
the induced map~$\Phi_{2k+2}^{\ast}$ is therefore an isomorphism 
in $R_{k+1}$-cohomology in all degrees. 
As the~$R_{k+1}$-cohomologies of $M(2k+2)$ and $N(2k+2)$ are  
concentrated in even degrees, the universal coefficient theorem implies that $R_{k+1}$-cohomology and 
$R_{k+1}$-homology are dual, so $\Phi_{2k+2}$ induces an isomorphism in $R_{k+1}$-homology. 
Hence $\Phi_{2k+2}$ is a $\mathcal{P}_{k+1}$-local homotopy equivalence by Whitehead's Theorem. 

Further, if $\pi_{2k+2}(\mathcal{Z}_P)\otimes\mathbb{Q}=0$ then  
Lemma~\ref{rationalextension} applied to~(\ref{skeldgrm}) implies that
$(\Phi_{2k+2})_0\simeq e_{2k+2}$.  
\end{proof} 

Corollary~\ref{keypropcor} leads to an inductive argument producing a local homotopy 
equivalence $M_{\mathcal{P}_n}\simeq N_{\mathcal{P}_n}$. 

\begin{theorem} 
   \label{induction} 
Let $N$ be a $2n$-dimensional quasitoric manifold associated with a simple polytope $P$ such that each $\pi_{2k+2}(\mathcal{Z}_P)\otimes\mathbb{Q}=0$
   for $1\leq k<n-1$. If $M$ is a quasitoric manifold with a graded ring isomorphism $H^{\ast}(M;\Z)\cong H^{\ast}(N;\Z)$,
then there is a $\mathcal{P}_{n}$-local homotopy equivalence~$M_{\mathcal{P}_n}\simeq N_{\mathcal{P}_n}$. 
\end{theorem}

\begin{proof} 
First observe that the dimension of $M$ is $2n$ by hypotheses.
The proof proceeds by induction on skeletons. 
The base case is the integral homotopy equivalence 
\(\namedright{M(2)}{\Phi_{2}}{N(2)}\) 
in Lemma~\ref{skel2QZ} such that $(\Phi_2)_0\simeq e_2$. Note that $\mathcal{P}_{1}$ is the set of 
all primes so $R_{1}=\mathbb{Z}$, implying that 
$M(2)=M_{\mathcal{P}_{1}}(2)$, $N(2)=N_{\mathcal{P}_{1}}(2)$, 
and $\Phi_{2}$ can be viewed as a $\mathcal{P}_{1}$-local homotopy equivalence 
\(\namedright{M_{\mathcal{P}_{1}}(2)}{\Phi_{2}}{N_{\mathcal{P}_{1}}(2)}\). 
The inductive step is given by 
Corollary~\ref{keypropcor}. The induction stops once the $\mathcal{P}_{n}$-local 
homotopy equivalence~
\(\lnamedright{M_{\mathcal{P}_{n}}(2n)}{\Phi_{2n}}{N_{\mathcal{P}_{n}}(2n)}\) 
has been constructed. As 
$M_{\mathcal{P}_{n}}(2n)\simeq M_{\mathcal{P}_{n}}$ and $N_{\mathcal{P}_{n}}(2n)\simeq N_{\mathcal{P}_{n}}$, we obtain a $\mathcal{P}_{n}$-local homotopy 
equivalence $M_{\mathcal{P}_{n}}\simeq N_{\mathcal{P}_{n}}$.

Note that the condition $\pi_{2k+2}(\mathcal{Z}_{P})\otimes\mathbb{Q}=0$ in Corollary~\ref{keypropcor} is used in the $(k+1)$-step of the induction, so as the induction stops when $k=n$, the condition $\pi_{2n+2}(\mathcal{Z}_{P})\otimes\mathbb{Q}=0$ is not needed, which explains why this is omitted from the list of hypotheses.
\end{proof} 

We now prove the main result.  

\begin{proof}[Proof of Theorem~\ref{main}] 
By hypothesis the simple polytope is $P=\prod_{i=1}^{\ell}\Delta^{d_{i}}$, where $d_{i}\geq 1$. In this case, by~\cite[Proposition 6.4]{BP} there is a 
diffeomorphism $\mathcal{Z}_{P}\cong\prod_{i=1}^{d} S^{2d_{i}+1}$. Since an odd dimensional sphere $S^{2\ell+1}$ is rationally the Eilenberg-MacLane space 
$K(\mathbb{Q},2\ell+1)$, we have $\pi_{2k+2}(\mathcal{Z}_{P})\otimes\mathbb{Q}= 0$ for all $k\geq 1$. 
Therefore, by Theorem~\ref{induction}, 
the graded ring isomorphism $H^{\ast}(M)\cong H^{\ast}(N)$ implies that there is an $R_{n}$-local 
homotopy equivalence $M_{\mathcal{P}_n}\simeq N_{\mathcal{P}_n}$. It remains to identify the primes that have been inverted in $R_{n}$. 

By definition, $d$ is the minimum of $\{d_{1},\ldots,d_{\ell}\}$. For any prime $p$, the least 
dimensional $p$-torsion class in $\pi_{\ast}(S^{2t+1})$ occurs in dimension $2t+2p-2$. Thus if the 
least dimensional sphere in the product decomposition for $\mathcal{Z}_{P}$ is $S^{2d+1}$ and 
the highest dimension for which a prime in $\pi_{2k+1}^{\mbox{\scriptsize Tor}}(\mathcal{Z}_{P})$ 
must be inverted is $k=n-1$ then we need to localize away from all primes satisfying 
$2d+2p-2\leq 2n-1$, implying that $p\leq n-d$. That is, all primes $p\leq n-d$ have been 
inverted in $R_{n}$. 
\end{proof}

\newcommand{\etalchar}[1]{$^{#1}$}

\end{document}